\documentclass{amsart}[12pt]
\usepackage{graphicx}

\newtheorem{theorem}{Theorem}[section]
\newtheorem{lemma}[theorem]{Lemma}

\theoremstyle{definition}

\theoremstyle{remark}

\numberwithin{equation}{section}

\newcommand{\ud}{\mathrm{d}}
\newcommand{\w}[1]{\widehat{#1}}

\begin{document}

\title{Critical case stochastic phylogenetic tree model via the Laplace transform}
\author{Krzysztof Bartoszek}
\address{Krzysztof Bartoszek \\ Chalmers University of Technology and the University of Gothenburg (corresponding author)}
\curraddr{Mathematical Sciences, Chalmers University of Technology and the University of Gothenburg 412 96 G\"oteborg, Sweden. }
\email{krzbar@chalmers.se}

\author{Micha\l\ Krzemi\'nski}
\address{Micha\l\ Krzemi\'nski \\ Gda\'nsk University of Technology, Polish Academy of Sciences}
\curraddr{Department of Probability Theory and Biomathematics, Faculty of Applied Mathematics and Technical Physics, Gda\'nsk University of Technology, 80-233 Gda\'nsk, Poland Institute of Mathematics, Polish Academy of Sciences, 00-956 Warszawa, Poland.}
\email{ mkrzeminski@mif.pg.gda.pl}



\begin{abstract}
Birth--and--death models are now a common mathematical tool to
describe branching patterns observed in real--world phylogenetic
trees. Liggett and Schinazi (2009) is one such example. The authors
propose a simple birth--and--death model that is compatible
with phylogenetic trees of both influenza and HIV, depending on
the birth rate parameter. An interesting special case of this model is
the critical case where the birth rate equals the death rate. 
This is a non--trivial situation and to study its asymptotic
behaviour we employed the Laplace transform.
With this we correct the proof of Liggett and Schinazi (2009)
in the critical case.
\end{abstract}

\maketitle

\section{Introduction}
Different viral types have phylogenetic trees exhibiting different branching properties,
with influenza and HIV being two extreme examples. In the influenza tree a single type
dominates for a long time with other types dying out quickly until suddenly a
new type completely takes over and the old type dies out. The HIV phylogeny
is the complete opposite, with a large number of co--existing types.

In \cite{TLigRSch2009} a stochastic model is described and depending on the choice
of parameters it can exhibit both types of dynamics. We briefly describe the model
after \cite{TLigRSch2009}. We only keep track of the number of different
viral types at each time point $t$. Let $N(t)$ denote the number
of distinct viral types at time $t$. In the nomenclature of
phylogenetics $N(t)$ counts the number of different
species alive at time $t$.
At each time point the birth rate
is $\lambda N(t)$ and the death rate is $N(t)$. If there is only one type
alive then it cannot die.
Clearly $N(t)$ is a Markov chain with discrete state space and continuous time.
Each virus type is described by a fitness value
that is randomly chosen at its birth. If a death event occurs
the type with smallest fitness dies. This means that only the fitness
ranks matter and so the exact distribution of a virus' fitness
will not play a role.

The main result of \cite{TLigRSch2009} is the asymptotic behaviour
of the dominating type, whether it is expected to remain the same for
long stretches of time or change often. This is summarized in Theorem 1 \cite{TLigRSch2009},
\begin{theorem}[Theorem 1, \cite{TLigRSch2009}]
Take $\alpha \in (0,1)$. If $\lambda \le 1$ then
\begin{displaymath}
\lim \limits_{t \rightarrow \infty} \mathrm{P}(maximal~types~at~times~\alpha t~and~t~are~the~same) = \alpha,
\end{displaymath}
while if $\lambda > 1$ then this limit is $0$.
\end{theorem}
The proof of this theorem is based on considering successive visits to the state $1$,
in particular denote $\tau_{1}, \tau_{2}, \ldots$ to be the (random) times between
visits of the chain to $N(t)=1$ and $T_{n}:=\tau_{1}+\ldots+\tau_{n}$.
In \cite{TLigRSch2009} the latter random variable is represented as
\begin{displaymath}
T_{n} = \sum_{i=1}^{n} X_{i} + \sum_{i=1}^{n} H_{i},
\end{displaymath}
where $X_{i}$ are independent mean $1$ exponential random variables and $H_{i}$ are the
(independent) hitting times of the state $1$ conditional on starting in state $2$.
Descriptively $H_{i}$ is the $i$th return from state $2$ to the state of one virus type alive,
since from $1$ the chain has to jump to two types. The Markovian nature
of the process ensures that the $H_{i}$s are independent and identically distributed for distinct $i$s.
In the proof of Theorem 1 it is stated that the cumulative distribution function
of $H_{i}$, $F(t)$, satisfies,
\begin{equation}\label{eqLigIntFt}
\int \limits_{0}^{F(t)} \frac{1}{1+s^{2}-2s} \ud s = t,
\end{equation}
solved uniquely for,
\begin{equation}\label{eqLigFt}
F(t) = \frac{t}{1+t}.
\end{equation}
This gives the asymptotic behaviour (Lemma $3$ \cite{TLigRSch2009}),
\begin{equation}\label{eqLigNt}
\lim_{n \rightarrow \infty} \frac{T_{n}}{n\log(n)}= 1 ~~~~ \mathrm{in~probability},
\end{equation}
from which the result of Theorem 1 is derived when $\lambda=1$.

However Eq. (\ref{eqLigIntFt}) does not take into account
that this model differs from a classical birth--death model where $0$ is the
absorbing state. We illustrate this in Fig. \ref{fgBD}. We can easily
re--numerate the state values, but the intensity values will differ between the two models.
\begin{figure}
\begin{center}
\includegraphics{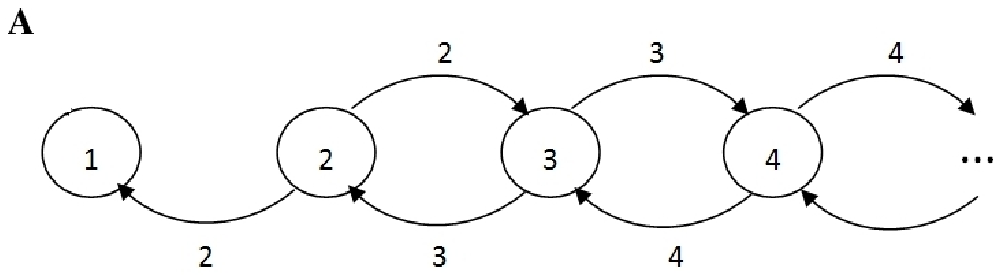} \\
\includegraphics{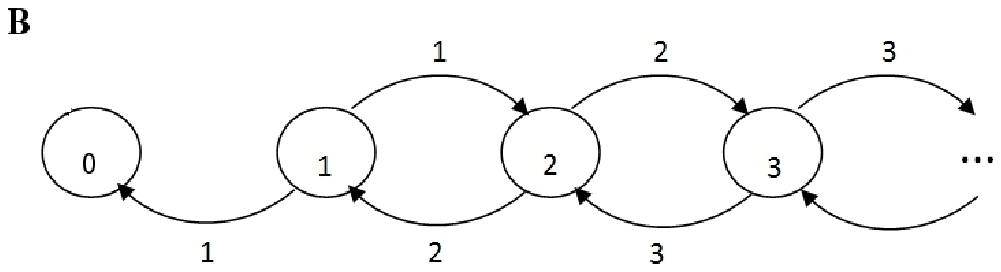}
\caption{A: Depiction of the Markov chain model described in \cite{TLigRSch2009}. \mbox{B: Depiction} of a classical
Markov chain model for which Eq. (\ref{eqLigIntFt}) would be correct. The numbers inside the circles
are the states (counting the number of virus type) and the numbers above
and below the arrows are the birth and death rates respectively of the state
from which they come out.}\label{fgBD}
\end{center}
\end{figure}
Correcting for this difference in intensity values one
will still get the same asymptotics as in Eq. (\ref{eqLigNt})
and hence the same result as in \mbox{Theorem 1}
but with a more complicated proof. Below we present a correct
proof of Lemma $3$ \cite{TLigRSch2009}
in the case of $\lambda=1$,  based on Lemmas \ref{lemRenewal} and \ref{lemFtasympt}.
From this the statement of Theorem  $1$ of \cite{TLigRSch2009} follows.

\section{Auxiliary lemmas}
\begin{lemma}\label{lemRenewal}
Let $P_{1}(t)\equiv P(H_1 \leq t) \equiv F(t)$. Then $P_{1}(t)$ solves the renewal equation
\begin{equation}\label{EqRenewal}
P_{1}(t)  =  \frac{2t}{(1+t)^{3}} + \int_{0}^{t}\frac{2}{(1+(t-y))^{3}}P_{1}(y) \ud y.
\end{equation}
\end{lemma}
\begin{proof}
In panel A of Fig. \ref{fgBD} we can see a representation of the studied Markov chain
on the state space $S=1,2,3,\ldots$. Due to the $H_{i}$s being independent for
different $i$s we can study the distribution of $H_{1}$ and treat $1$ as an absorbing
state. Let $P_{n}(t)$ denote the probability of being in state $n$ at time $t$,
when one starts in state $2$ at time $0$.
The system of differential equations describing the probabilities is,
\begin{equation}
\left\{\begin{array}{ll}\label{p123}
P'_1(t)&=2P_2(t)\\
P'_2(t)&=-4 P_2(t)+3P_3(t)\\
P'_n(t)&=-2nP_n(t)+(n+1)P_{n+1}(t)+(n-1)P_{n-1}(t),~~n \ge 3\\
\end{array}\right.
\end{equation}
with initial conditions,
\begin{equation}
P_1(0)=0, \qquad P_2(0)=1, \qquad P_3(0)=\ldots=P_n(0)=\ldots=0.
\end{equation}
Let $P(s,t)$ denote the generating function of the sequence $\{P_{n}(t)\}_{n=1}^{\infty}$,
i.e.
\begin{displaymath}
P(s,t)=\sum_{n=1}^{\infty}P_{n}(t)s^{n}.
\end{displaymath}
Taking first derivatives we get the partial differential equation,
\begin{equation}\label{partdiffeq}
\frac{\partial P(s,t)}{\partial t}=-(s-1)^2P_1(t)+(s-1)^2\frac{\partial P(s,t)}{\partial s},
\end{equation}
with initial conditions
\begin{equation}
\frac{\partial P(s,t)}{\partial s}\mid_{s=0\atop t=t}=P_1(t) \quad P(s,t)\mid_{s=s\atop t=0}=\sum_{n=1}^{\infty}P_n(0)s^n=s^2.
\end{equation}
Following \S$2.1$ \cite{Evans} and using the substitution $z(x)=P(h(s,x),t+x)$ with
\begin{displaymath}
h(s,x) :=  1 + \frac{1}{x+\frac{1}{s-1}} \quad s \neq 1,
\end{displaymath}
we arrive at,
\begin{displaymath}
P(s,t) - \left(\frac{s-(s-1)t}{1-t(s-1)}\right)^{2} = z(0)-z'(-t) = \int_{-t}^{0} z'(x) \ud x = 
\int_{0}^{t} \frac{-1}{\left(y-t+\frac{1}{s-1} \right)^{2}} P_{1}(y)  \ud y.
\end{displaymath}
Evaluating the derivative of both sides with 
respect to $s$ at $0$ we find that the function $P_{1}(t)$ must satisfy the following integral equation
(we can recognize it as a renewal equation),
\begin{equation}
\frac{\partial P}{\partial s}(0,t) = P_1(t)  =  \frac{2t}{(1+t)^{3}} + \int_{0}^{t}\frac{2}{(1+(t-y))^{3}}P_1(y) \ud y.
\end{equation}
\end{proof}

\begin{lemma}\label{lemFtasympt}
$F(t) \equiv P_{1}(t)$, the solution of the renewal equation (\ref{EqRenewal}),
has the following properties,
\begin{equation}\label{mnh}
\int_0^{h}tF'(t)\ud t \sim \log(h) ~~ \mathrm{as~} h \rightarrow \infty,
\end{equation}
\begin{equation}\label{snh}
\int_0^{h}t^2F'(t)\ud t \sim h ~~ \mathrm{as~} h \rightarrow \infty.
\end{equation}
\end{lemma}
\begin{proof}
The proof is based on Tauberian theory and we refer the reader to \cite{Feller,Korevaar} for details on this.
Another approach would be to study the asymptotic behaviour of $F'(t)$ by renewal theory results
(see e.g. \cite{PEmbEOme1984,TLig1989}). 
The Laplace transform of a density function $f(x)$, denoted $\w{f(x)}(s)$ is,
\begin{displaymath}
\w{f(x)}(s) = \int_0^\infty e^{-xs} f(x) \ud x.
\end{displaymath}
We will use the following theorem from \cite{Feller}, Theorem $2$ \S XIII.5,
\begin{theorem}[Theorem $2$ \S XIII.5, \cite{Feller}]\label{tauberian}
If $L$ is slowly varying at infinity and $0 \leq \rho < \infty$, then each of the relations
\begin{equation}
\w{f(t)}(s) = \int_0^\infty f(t) \exp(-st) \ud t \sim s^{- \rho} L(\frac1s), \qquad s \rightarrow 0,
\end{equation}
\begin{equation}
F(t) = \int_0^t f(u)\ud u \sim \frac1{\Gamma(\rho +1)} t^\rho L(t), \qquad t \rightarrow \infty
\end{equation}
implies the other.
\end{theorem}
$F(t)\equiv P_{1}(t)$ defined as the solution to the renewal equation (\ref{EqRenewal}) after
differentiating will satisfy
\begin{equation}\label{diffF}
F'(t) = \frac{2-4t}{(1+t)^{4}} + \int_{0}^{t}\frac{2}{(1+t-y)^{3}}F'(y) \ud y
\end{equation}
We calculate the Laplace transforms of $F'(t)$, $tF'(t)$ and $t^2F'(t)$,
\begin{align}\label{wdiffF}
\w{F'(t)}(s) &= \w{\frac{2-4t}{(1+t)^{4}}}(s) +\w{\frac{2}{(1+t)^{3}}}(s) \cdot \w{F'(t)}(s) \\\label{wtdiffF}
\w{tF'(t)}(s) &= \w{\frac{(2-4t)t}{(1+t)^{4}}}(s) + \w{\frac{2(t)}{(1+t)^{3}}}(s)\cdot \w{F'(t)}(s) + \w{\frac{2}{(1+t)^{3}}}(s)\cdot \w{tF'(t)}(s)\\\label{wttdiffF}
\w{t^2F'(t)}(s) &= \w{\frac{(2-4t)t^2}{(1+t)^{4}}}(s) + \w{\frac{2(t)^2}{(1+t)^{3}}}(s) \cdot \w{F'(t)}(s)+ \w{\frac{4t}{(1+t)^{3}}}(s)\cdot \w{tF'(t)}(s) + \w{\frac{2}{(1+t)^{3}}}(s) \cdot \w{t^2F'(t)}(s).
\end{align}
We are interested in the behaviour of the transforms as $s \rightarrow 0$, and for this
we will use the well known property (verifiable by the de L'H\^ospital rule) of the exponential integral,
\begin{displaymath}
\int_s^\infty \frac{\exp(-u)}u  \ud u \sim -\log(s),
\end{displaymath}
to arrive at,
\begin{align}
\w{tF'(t)}(s) = \cfrac{ \w{\frac{(2-4t)t}{(1+t)^{4}}}(s) + \w{\frac{2t}{(1+t)^{3}}}(s)\cdot \w{F'(t)}(s) }{ 1- \w{\frac{2}{(1+t)^{3}}}(s) }\sim \log(\frac1s),\qquad s \rightarrow 0\\
\w{t^2F'(t)}(s) = \cfrac{ \w{\frac{(2-4t)t^2}{(1+t)^{4}}}(s) + \w{\frac{2t^2}{(1+t)^{3}}}(s) \cdot \w{F'(t)}(s)+ \w{\frac{4t}{(1+t)^{3}}}(s)\cdot \w{tF'(t)}(s) }{ 1- \w{\frac{2}{(1+t)^{3}}}(s) } \sim \frac1s,\qquad s \rightarrow 0\\
\w{t(1-F(t))}(s) = \frac{ \w{\frac{t}{(1+t)^2}}(s)- \w{\frac{2t^2}{(1+t)^{3}}}(s) + \w{\frac{2t}{(1+t)^{3}}}(s)\cdot \w{(1-F(t))}(s)} {1-\w{\frac{2}{(1+t)^{3}}}(s)} \sim \frac1s\qquad s \rightarrow 0. \label{twF}
\end{align}
As  both the constant function $(s^{0})$ and $\log(1/s)$ are slowly varying functions for $s \rightarrow 0$
the Tauberian theorem allows us to conclude that,
\begin{align}
\int_0^{h} tF'(t) \ud t \sim \log(h)\\
\int_0^{h} t^2F'(t) \ud t \sim h. \label{eqt2F}
\end{align}
\end{proof}

\section{Proof of Lemma 3 \cite{TLigRSch2009}}
We will now use Lemmas \ref{lemRenewal} and \ref{lemFtasympt} to prove Lemma 3 from \cite{TLigRSch2009}.
Define as there,
\begin{displaymath}
m_{n} := \int_{0}^{\rho_{n}}tF'(t)\ud t \qquad \mathrm{and} \qquad s_{n} := \int_{0}^{\rho_{n}}t^{2}F'(t)\ud t,
\end{displaymath}
with $\rho_{n}:=n\sqrt{\log(n)}$.
By Lemma \ref{lemFtasympt} we know that,
\begin{displaymath}
\begin{array}{ccc}
m_{n} \sim \log (\rho_{n}) \sim \log(n) & \mathrm{and} & s_{n} \sim \rho_{n}.
\end{array}
\end{displaymath}
We now need to check how $n(1-F(\rho_{n}))$ behaves asymptotically. We do not know what $F(\rho_{n})$
is but using the Tauberian theorem and Eq. (\ref{twF})
from the proof of Lemma \ref{lemFtasympt} we get that,
\begin{equation}
\int_0^t u(1-F(u)) \ud u \sim t, \qquad t\rightarrow \infty.
\end{equation}
Therefore using integration by parts and Eq. (\ref{eqt2F}),
\begin{align}
1\sim \frac{\int_0^t u(1-F(u)) du}{t}&=\frac{\frac{t^2}2(1-F(t))}{t} +\frac{\int_0^t\frac{u^2}2F'(u)du}{t}, \qquad t\rightarrow \infty\notag\\
1\sim&\frac12(t(1-F(t))) +\frac12, \qquad t\rightarrow \infty\notag,
\end{align}
and so we arrive at
\begin{equation}
\lim \limits_{t\rightarrow \infty}t(1-F(t))\rightarrow 1.
\end{equation}
The rest of the proof is a direct repeat of the one in \cite{TLigRSch2009} 
and so we get (as in \cite{TLigRSch2009}) that $T_{n}/(n\log(n))$ tends to $1$
in probability implying Theorem 1 \cite{TLigRSch2009} for $\lambda=1$.

If we applied the same chain of reasoning to the model of panel B in Fig. \ref{fgBD}
starting off with the system of differential equations, the analogue of 
Eq. (\ref{partdiffeq}) would be a homogeneous partial differential equation,
\begin{equation}\label{partdiffeqB}
\frac{\partial P(s,t)}{\partial t}=(s-1)^2\frac{\partial P(s,t)}{\partial s},
\end{equation}
with initial conditions
\begin{equation}
\frac{\partial P(s,t)}{\partial s}\mid_{s=0\atop t=t}=P_1(t) \quad P(s,t)\mid_{s=s\atop t=0}=s,
\end{equation}
in agreement with $P_{1}(t)=t/(t+1)$.
It would therefore be an interesting problem to see what conditions are necessary
on the nonhomogeneous part of Eq. (\ref{partdiffeq}) 
to still get the same asymptotic behaviour of the Markov chain
and what underlying model properties do these conditions imply.

\section*{Acknowledgments}
We are grateful to Wojciech Bartoszek and Joachim Domsta for many helpful
comments, insights and suggestions.
K.B. was supported by the Centre for Theoretical Biology at the University of Gothenburg, 
Stiftelsen f\"or Vetenskaplig Forskning och Utbildning i Matematik
(Foundation for Scientific Research and Education in Mathematics), 
Knut and Alice Wallenbergs travel fund, Paul and Marie Berghaus fund, the Royal Swedish Academy of Sciences,
Wilhelm and Martina Lundgrens research fund
and \"Ostersj\"osamarbete scholarship from Svenska Institutet (00507/2012).

\end{document}